\documentclass[11pt]{amsart}
\usepackage{amsthm}
\usepackage{amssymb,latexsym,graphics,enumerate}
\usepackage[mathscr]{eucal}
\usepackage{amsmath,amsfonts,amsthm,amssymb}
\usepackage{color}
\usepackage{hyperref}
\usepackage[left=1in,right=1in,top=1in,bottom=1in]{geometry}

\usepackage{graphicx}
\numberwithin{equation}{section}

\newcommand{\rr}{\mathbb{R}}

\newcommand{\be}{\begin{eqnarray*}}
\newcommand{\bel}{\begin{eqnarray}}
\newcommand{\ee}{\end{eqnarray*}}
\newcommand{\eel}{\end{eqnarray}}
\newcommand{\ba}{\begin{aligned}}
\newcommand{\ea}{\end{aligned}}
\newcommand{\de}{\Delta}

\newcommand{\na}{\nabla}
\newcommand{\ep}{\epsilon}

\newcommand{\pa}{\partial}

\newcommand{\nb}{\nonumber}
\newcommand{\CC}{\frac{K}{2}}
\newcommand{\CI}{\frac{2}{K}}
\newtheorem{thm}{Theorem}

\newtheorem{lem}{Lemma}

\newtheorem{rmk}{Remark}

\numberwithin{rmk}{section}
\numberwithin{lem}{section}
\numberwithin{thm}{section}

\newcommand\R{{\mathbb R}}

\title{Boundary layer models of the Hou-Luo scenario}
\date{\today}

\author{Siming He} \thanks{simhe@math.duke.edu, Department of Mathematics, Duke University}
\author{Alexander Kiselev} \thanks{kiselev@math.duke.edu, Department of Mathematics, Duke University}

\begin{document}
\begin{abstract}
Finite time blow up vs global regularity question for 3D Euler equation of fluid mechanics is a major open problem.
Several years ago, Luo and Hou \cite{HouLuo14} proposed a new finite time blow up scenario based on extensive numerical simulations.
The scenario is axi-symmetric and features fast growth of vorticity near a ring of hyperbolic points of the flow located at the
boundary of a cylinder containing the fluid. An important role is played by a small boundary layer where intense growth is observed.
Several simplified models of the scenario have been considered, all leading to finite time blow up \cite{CKY15,CHKLVY17,HORY,KT1,HL15,KY1}.
In this paper, we propose two models that are designed specifically to gain insight in the evolution of fluid near the hyperbolic stagnation point of the
flow located at the boundary. One model focuses on analysis of the depletion of nonlinearity effect present in the problem. Solutions to this model are shown to be
globally regular. The second model can be seen as an attempt to capture the velocity field near the boundary to the next order of
accuracy compared with the one-dimensional models such as \cite{CKY15,CHKLVY17}. Solutions to this model blow up in finite time.
\end{abstract}

\maketitle

\section{Introduction}
The problem of finite time blow vs global regularity in solutions of the 3D Euler equation of fluid mechanics is one of the major
open questions of applied analysis. Several years ago, Hou and Luo \cite{HouLuo14} produced a very careful and convincing numerical simulation
proposing a new scenario for finite time blow up. The simulation is axi-symmetric, and takes place in an infinite vertical cylinder with rigid
boundary. The boundary conditions are no penetration; the solution is periodic in vertical direction and obeys additional symmetries with respect
to $z=0$ plane. Although the symmetries make the Hou-Luo scenario look special, a recent experimental paper classifying the structure of regions
of ``extreme dissipation" in turbulent fluid \cite{nsexp} has found that most of these are formed by colliding masses of fluid with a hyperbolic
point of the flow in between - geometry related to the Hou-Luo scenario. Thus the Hou-Luo scenario can be thought of as an idealized blueprint for
a fairly robust small scale creation and energy dissipation mechanism present in turbulence.

It is well-known \cite{MB} that the 2D Boussinesq system is a good proxy for the 3D axi-symmetric Euler equation away from the axis, and it has a
bit more compact form. Hence our starting point for the derivation of the models will be the 2D inviscid Boussinesq system set in a half-plane $\R^2_+:$
\begin{align}
\left\{\begin{array}{rrrr}\ba\pa_t\omega+(u\cdot \na)\omega=&\pa_{x_1}\rho;\\
\pa_t\rho+(u\cdot\na)\rho=&0;\\
u=\na ^{\perp}(-\de_D)^{-1} &\omega;\\
\omega(t=0,\cdot)=\omega_0&(\cdot),\quad \rho(t=0,\cdot)=\rho_0(\cdot),
\ea\end{array}\right.\label{Boussinesq}
\end{align}
where $\na^{\perp}=(-\pa_{x_2},\pa_{x_1})$,
 and $\de_D$ stands for the Laplacian operator with Dirichlet boundary condition on the boundary $x_2=0.$
The Biot-Savart law in \eqref{Boussinesq} corresponds to the no penetration boundary condition $u_2(x_1,0) =0.$
The finite time blow-up versus global regularity problem for \eqref{Boussinesq} subject to smooth initial data is, naturally, open,
and has been listed among Yudovich's eleven great problems of mathematical hydrodynamics \cite{Y11}.
In the Boussinesq setting, to arrive at an analog of the Hou-Luo scenario, one takes the initial data $\rho_0$ even and $\omega_0$ odd with respect to the $x_1$ variable.
The support of the initial data lies away from the $x_2$ axis but contains some part of the $x_1$ axis.

In the recent years, there has been much work on analyzing the Hou-Luo scenario and similar settings. Kiselev and Sverak \cite{KS14} deployed this
geometry of the initial data on the unit disk to construct a solution of the 2D Euler equation whose gradient grows at a double exponential rate
in time. The 2D Euler equation in vorticity form can be obtained from \eqref{Boussinesq} by setting $\rho_0(x)=0.$ The 2D Euler equation
is well known to be globally regular - but the double exponential rate is also known as the fastest possible growth rate of derivatives \cite{MB,MP,Wolibner}.
This result confirms the proclivity of the Hou-Luo initial data set up to lead to extreme small scale creation. One of the key discoveries of
\cite{KS14} has been an efficient formula for the Biot-Savart law in the neighborhood of the origin - a hyperbolic point on the boundary near which
explosive growth of derivatives happens. The formula takes form
\begin{equation}\label{mainlemma}
u_{i}(x,t) = \frac{x_i (-1)^i}{4\pi} \int_{Q(x)}\frac{y_1y_2}{|y|^4}\omega(y,t)\,dy + x_i B_i(x,t), \,\,\,\|B_i\|_\infty \leq C \|\omega_0\|_\infty,
\end{equation}
where $Q(x)$ is a region $y_1 \geq x_1,$ $y_2 \geq x_2$ in the half-disk $D^+$ where $x_1 \geq 0;$ the origin is located at the lowest point of the disk.
The formula isolates the ``main term" that has, asymptotically, log-Lipshitz behavior in the distance from vorticity support to the
origin. It is this behavior that saturates bounds of the Yudovitch theory \cite{Yudth,MP} and makes double exponential growth possible.

In attempting to pass from the analysis of 2D Euler solutions \cite{KS14} to Boussinesq \eqref{Boussinesq}, one encounters several difficulties.
Perhaps the main issue is that the vorticity is uniformly bounded for the 2D Euler solutions, while for Boussinesq system vorticity may grow.
This makes the error terms in the fluid velocity representation \eqref{mainlemma} potentially large and competitive with the main term.
Sufficient control of the fluid velocity, on the other hand, appears to be crucial in achieving understanding of dynamics necessary to
prove finite time blow up. A number of simplified models have been considered \cite{CKY15,CHKLVY17,DKX,HORY,KT1,HL15,KY1}; in all these models
the Biot-Savart law is replaced by a version that is simpler to control. Some of these models also feature the forcing term $\partial_{x_1} \rho$
replaced by sign definite mean field expression $\rho/x_1.$

There is an exciting recent series of works by Elgindi and Elgindi and Jeong \cite{E1,EJ1,EJ2} that is also related to Hou-Luo scenario.
In papers \cite{EJ1,EJ2}, the authors consider 2D Boussinesq system and 3D axi-symmetric Euler equation respectively in domains with corners
or wedges; they prove local well-posedness theorem for a class of rough initial data, and then show finite time blow up in a sense that a stronger
singularity forms. The geometry of the construction features a hyperbolic point of the flow similar to the Hou-Luo scenario. In a recent work
\cite{E1}, Elgindi carries out a similar construction in the whole space, without a boundary. The initial vorticity data is H\"older regular,
with just one point where it is not smooth. The H\"older exponent $\alpha$ that tracks regularity of the initial data plays a role of a small
parameter which essentially allows one to identify the ``main term" with the blow up dynamics accessible to analysis.
A related idea has been used in \cite{KRYZ} to construct blowing up solutions of modified SQG equation. Here the geometry is identical to the one
we consider in this paper: a half-plane with a fixed hyperbolic point on the boundary created by symmetries. The Biot-Savart terms conducive to and
 opposing blow up have the same order, and the exponent $\alpha$ modulating singularity of the Biot-Savart law plays a role of the small parameter
that allows to establish dominance of the blow up term.

In the problem of smooth data for the 2D inviscid Boussinesq system (and 3D axi-symmetric Euler), there is no obvious small parameter
that can be leveraged in a similar way. Hence the problem of blow up vs global regularity for smooth solutions remains open. The purpose of this paper
is to analyze two models aimed at the improved understanding of the Biot-Savart law in the smooth case. In the first model, the goal is to gain further
insight into mechanisms driving nonlinearity depletion in the geometry of Hou-Luo scenario. In this model, we discard part of the blow up-conducive
terms in the Biot-Savart law and show that the resulting dynamics is, in fact, globally regular. The mechanism behind this effect is thinning of the
filament of vorticity due to shear as the support of vorticity approaches the origin. This effect is definitely present in the real equation,
even though the balance of terms may be different - so understanding the workings of this phenomena is useful. The model is obtained by retaining only the
core of the nonlinearity depletion effect and is given by a family of coupled ODE
\begin{align}\label{non-local-ODE}
\partial_t G(x,t)=\int_0^1\frac{(y-K^{-1}x)e^{-G(y,t)}}{(y^2+e^{-2G(y,t)})^2}dy,\quad G(x,0)=0.
\end{align}
Here $e^{-G(x,t)}$ can be thought of as an $x_1$ location of the front of vorticity region (with $x=x_2$) advancing towards the origin at time $t,$
and $K$ is a universal constant near $1.$ 
We will sketch the derivation of \eqref{non-local-ODE} in Section~\ref{nlodederiv}
below. Finite time blow up then corresponds to $G(x,t)$ becoming infinite in finite time - this would indicate collision
of vorticity masses of different sign in the original system. Observe that due to the structure of the model \eqref{non-local-ODE},
$G(x,t)$ is a linear function of $x,$ and we can set $G(x,t) = A(t)-B(t)x.$
Our main result is the following:
\begin{thm}\label{Thm_1}
Consider the model \eqref{non-local-ODE}. If $1\leq K \leq 1.3$, the solution to the system \eqref{non-local-ODE} is global in time. For large times, the following asymptotic behavior holds:
\begin{align}
e^{A(t)}\approx &B(t)^{\frac{K}{3}},\\
B(t)\approx &t^{1/(2-K)}
\end{align}
for $t \geq 1.$
\end{thm}
\begin{rmk}
1. The $K=1$ case corresponds to the original Biot-Savart law in the Boussinesq equation \eqref{Boussinesq}. We extend the range of $K$ upward to study the stability of the suppression effect under perturbation of the equation. \\
2. The following notational convention will be used in the paper. Assume that $A,B$ are two positive quantities. We use the notation $A\lesssim B$ or $A\gtrsim B$ if there exists a universal constant $C$ such that $A\leq CB$ or $A\geq \frac{1}{C}B$, respectively. The notation $A\approx B$ is applied if there exists a universal constant $C$ such that $\frac{1}{C}B\leq A\leq CB$.
\end{rmk}
The picture described by Theorem~\ref{Thm_1} is that of a filament of vorticity reaching towards the origin - if it ever arrives there, this would signify blow up.
But the filament gets thinned by shear,
which is reflected in growth of $B(t),$ and due to this fact it only gets to the origin in an infinite time and not at a particularly high power rate in time.

The second model that we consider here can be thought of as a variation of the original Hou-Luo model described already in \cite{HouLuo14} and shown to lead to singularity
formation in \cite{CHKLVY17}. The Hou-Luo model is derived under the assumption that vorticity remains constant in a thin boundary layer and zero elsewhere:
$\omega(x_1,x_2,t) = \omega(x_1,0,t) \chi_{[0,a]}(x_2),$ where $\chi_S$ is the characteristic function of the set $S.$ Assuming such ansatz one can derive
an effective Biot-Savart law that reduces the system \eqref{Boussinesq} to one dimension (essentially, to the boundary $x_2=0$). The assumption of vorticity that is
constant in a boundary layer, while a reasonable initial model, does not really fit with the numerical simulations \cite{HouLuo14}, and misses the effect of the
shear near the boundary that is likely crucial to understand in order to analyze the problem rigorously. We introduce a ``next order" model where we discard
$u_2$ component of the velocity, but do keep two dimensions and retain dependence of the $u_1$ component on $x_2,$ though restricting it to be linear.
The model system that we analyze takes form
\begin{subequations}\label{Higher_order_model}
\begin{align}
\pa_t \omega+u_1\pa_{x_1}\omega&=\frac{\rho}{x_1},\label{vorticity_equation}\\
\pa_t \rho+u_1\pa_{x_1}\rho&=0,\label{mass_transport_eq}\\
u_1(x_1,x_2,t)&=-x_1\iint_{[0,\infty)^2}\frac{y_1y_2\omega(y)}{|y|^4}dy+x_2\omega(x_1,0,t).\label{bs2mod}
\end{align}
\end{subequations}
Here we replace the derivative $\partial_{x_1}\rho$ with the mean field approximation, similarly to \cite{HORY,KT1}. In the Biot-Savart law \eqref{bs2mod},
the first term on the right hand side gives the approximate value of $u(x_1,0,t):$ here we use the explicit but simplified form of the full Biot-Savart law,
in the spirit of \cite{KS14,CKY15}. The second term on the right hand side of \eqref{bs2mod} takes advantage of a simple observation. Namely, on the boundary
$x_2=0$ we have $u_2(x_1,0,t)=0$ due to the no-penetration boundary condition. Hence also $\partial_{x_1}u_2(x_1,0,t)=0,$ and therefore
$\partial_{x_2} u_1(x_1,0,t) = \omega(x_1,0,t)$ for all $x_1,t.$ Thus the equation \eqref{bs2mod} becomes a linear in $x_2$ approximation
of $u_1$ near boundary. This is a natural first next order model that can be used to gain insight into the effect of the shear near boundary on singularity formation
in the Hou-Luo scenario.

The main results we obtain concerning the model \eqref{Higher_order_model} are summarized in the following theorem.
\begin{thm}\label{Thm2}
Consider the model \eqref{Higher_order_model}. There exist smooth initial data $(\omega_0,\rho_0)$ such that the solution blows up in finite time.
\end{thm}
\begin{rmk}
A similar result holds for the model with the original forcing term $\partial_{x_1} \rho.$
However, in this case one has to control the contribution of the negative vorticity. Due to the
approximation \eqref{bs2mod}, and specifically the global nature of the first term, this control is rather
artificially technical. Other than that, the treatment of the model involving the original forcing term
is quite close to the mean field one considered here.
\end{rmk}

The remaining part of the paper is organized as follows: in Section 2, we motivate and prove the well-posedness problem of the nonlocal ODE model \eqref{non-local-ODE}; in Section 3, we show
finite time blow up for \eqref{Higher_order_model}. 

\section{The ODE Model}
\subsection{Derivation and idea of the proof}\label{nlodederiv}
In this section, we explain the origin of the ODE model \eqref{non-local-ODE} and outline the proof of its global well-posedness.
To derive the model \eqref{non-local-ODE}, we will make several simplifications, including some basic assumptions on the initial data. Assume that the initial vorticity $\omega_0$ is odd
(in fact, we will take $\omega_0(x)=0$ for simplicity) and density $\rho_0$ is even in $x_1.$  These symmetries are preserved by the Boussinesq dynamics \eqref{Boussinesq}. Under this assumption, the Biot-Savart law, which determines velocity in terms of the vorticity, can be explicitly written as follows:
\begin{align}
u(x_1,x_2)=\frac{(\pa_{x_2},-\pa_{x_1})}{2\pi}\int_0^\infty dy_1\int_0^\infty dy_2(\log|x-y|-\log|x-\overline{y}|-\log |x-\widetilde{y}|+\log|x+y|)\omega(y),
\end{align}
where $\overline{y}=(y_1,-y_2)$, $\widetilde{y}=(-y_1,y_2)$ denote the reflections of a point $y$ with respect to the axes. We expect that singularity formation in
solutions of the equation \eqref{Boussinesq} is induced by the compression effect in the horizontal direction, so we simplify the model \eqref{Boussinesq} by setting the vertical component of the velocity to be zero, i.e., $u_2\equiv 0$.
Even though this violates the divergence free condition on the velocity, the created non-zero divergence effectively maintains the role of the vertical velocity
component $u_2.$ The horizontal component $u_1$ can be expressed as follows:
\begin{align}
u_1(x_1,x_2)=&\frac{1}{2\pi}\int_0^\infty \int_0^\infty dy\left(\frac{x_2-y_2}{|x-y|^2}-\frac{x_2+y_2}{|x-\overline{y}|^2}-\frac{x_2-y_2}{|x-\widetilde{y}|^2}+\frac{x_2+y_2}{|x+y|^2}\right)\omega(y)\nonumber\\
=&\frac{1}{2\pi}\int_0^\infty\int_0^\infty dy\bigg(\frac{4x_1y_1(x_2-y_2)\omega(y)}{((x_1-y_1)^2+(x_2-y_2)^2)((x_1+y_1)^2+(x_2-y_2)^2)}\nonumber\\
&-\frac{4x_1y_1(x_2+y_2)\omega(y)}{((x_1-y_1)^2+(x_2+y_2)^2)((x_1+y_1)^2+(x_2+y_2)^2)}\bigg)\nonumber\\
=&: -x_1 I(x,t).\label{defn_I}
\end{align}
Later we will choose the initial data such that the vorticity will be supported in a compact set close to but away from  the origin, and it will be safe to replace the integral domain to $Q=[0,1]^2$. Next we simplify the model \eqref{Boussinesq} further by changing the Buoyancy term from $\pa_{x_1}\rho$ to the mean field version $\frac{\rho}{x_1}$, i.e.,
the first equation in \eqref{Boussinesq} becomes
\begin{align}
\pa_t \omega+(u\cdot \na)\omega=\frac{\rho}{x_1}.\label{Model:Vorticity_eq}
\end{align}
The intuition behind this change is that since the vorticity is assumed to be supported away from the $x_2$ axis, the $x_1$-derivative can be  approximated by the difference quotient $\frac{\rho}{x_1}$ \cite{HORY,KT1}. With this modification, the vorticity in the first quadrant will stay positive under dynamics \eqref{Model:Vorticity_eq}.

Now let us track the evolution of the vorticity in the model \eqref{defn_I}, \eqref{Model:Vorticity_eq}. To this end, we first calculate the flow map $\Phi_t(x)$ associated to \eqref{defn_I} whose time evolution is defined as follows
\begin{align*}
\pa_t \Phi_t(x_0)=u(\Phi_t(x_0),t)=(-\Phi_t^{(1)}(x_0)I(\Phi_t(x_0),t),0),\quad \Phi_{t=0}(x_0)=x_0=(x_{1;0},x_{2;0}).
\end{align*}
Integration in time yields the exact form of the flow map
\begin{align}
\Phi_t(x_0)=\left(x_{1;0}\exp\left\{-\int_0^tI(\Phi_s(x_0),s)ds\right\},x_{2,0}\right)=:(x_{1;0}\exp{\{-G(x_0,t)\}},x_{2,0}).\label{Flow_map_0}
\end{align}
Later we call the function $G$ the profile of the solution since $e^{-G}$ tracks the position of the front of the vorticity. For the sake of simplicity, we abuse notation and use $\Phi_t(x)$ to denote the first component of the flow map $\Phi_t^{(1)}(x)$. Let $x$ denote the current location of the flow map \eqref{Flow_map_0}, i.e., $x:=\Phi_t(x_{1;0})$ and $x_{1;0}=\Phi_t^{-1}(x)$. Rewriting the equation \eqref{Model:Vorticity_eq} in the Lagrangian coordinate \eqref{Flow_map_0} and using the fact that the density is transported by the flow, i.e., $\rho(\Phi_t(x_0),t)\equiv \rho_0(x_0)$, we obtain
\begin{align}
\frac{d}{dt}\omega(\Phi_t(x_0),t)=\frac{\rho(\Phi_t(x_0),t)}{\Phi_t(x_0)}=\frac{\rho_0(x_0)}{x_{1;0}}\exp\left
\{\int_0^tI(\Phi_s(x_0),s)ds\right\}.\label{vorticity_1}
\end{align}
Integrating the equation \eqref{vorticity_1} and applying the inverse flow map the expression  \eqref{Flow_map_0} yields
\begin{align}
\omega(x,t)=&\omega_0(\Phi_t^{-1}(x))+\frac{\rho_0(\Phi_t^{-1}(x))}{\Phi_t^{-1}(x)}\int_0^t \exp\left\{\int_0^s I(\Phi_r(\Phi^{-1}_t(x)),r)dr\right\}ds\nonumber\\
=&\omega_0(\Phi_t^{-1}(x))+\frac{\rho_0(\Phi_t^{-1}(x))}{x_1}\int_0^t \exp\left\{-\int_0^tI(\Phi_r(\Phi_t^{-1}(x)),r)dr+\int_0^s I(\Phi_r(\Phi_t^{-1}(x)),r)dr\right\}ds\nonumber\\
=&\omega_0(\Phi_t^{-1}(x))+\frac{\rho_0(\Phi_t^{-1}(x))}{x_1}\int_0^t \exp\left\{-\int_s^t I(\Phi_r(\Phi_{t}^{-1}(x)),r)dr\right\}ds.\label{Vorticity_expression}
\end{align}
To clarify the possible blow-up suppression mechanism of the model, we make further simplifications. Since the last term under the integral in the final form of the
Biot-Savart law \eqref{defn_I} is always negative, its net effect is to enhance the drive towards the origin and therefore blow-up. We drop this last term in \eqref{defn_I}. Moreover, following the
intuition of \cite{KS14} we drop the $x_1$ variable in the denominator of the Biot-Savart law. 
Furthermore, we set the initial vorticity term in \eqref{Vorticity_expression} to zero as it is not essential for the phenomena that we would like to study. Substituting \eqref{Vorticity_expression} into the Biot-Savart law \eqref{defn_I} and using all the simplifications we mentioned above yields the following velocity law:
\begin{align}
u_1(x_1,x_2,t)
=-x_1\int_0^1\int_0^1\frac{(y_2-x_2)\rho_0(y_1\exp\{\int_0^tI(y_2,s)ds\})}{(y_1^2+(x_2-y_2)^2)^2}\int_0^t \exp\left\{-\int_s^t I(y_2,r)dr\right\} ds dy_1 dy_2.\label{I_x2_t}
\end{align}
Note that in the simplified velocity law, the function $I(\cdot,t)$ no longer depends on the first variable:
\[ I(x_2,t) = \frac{1}{2\pi}\int_0^\infty\int_0^\infty  \frac{4y_1(x_2-y_2)\omega(y)}{(y_1^2+(x_2-y_2)^2)^2}\,dy_1dy_2. \]
Suppose that initially, for $0 \leq x_2 \leq 1,$  the density is close to the characteristic function of an interval $[a,b]$, i.e., $\rho_0(x) \sim \chi_{[a,b]}(x_1).$
Let us first integrate the $y_1$ variable in \eqref{I_x2_t} and use the definition of the profile in \eqref{Flow_map_0} to obtain
\begin{align*}
\int_{a\exp\{-\int_0^tI(y_2,s)ds\}}^{b\exp\{-\int_0^tI(y_2,s)ds\}}\frac{1}{((x_2-y_2)^2+y_1^2)^2}dy_1\sim_{a,b}\frac{e^{-G}}{((x_2-y_2)^2+e^{-2G})^2}.
\end{align*}
To get a model amenable to precise analysis, we will discard the nonlocal in time factor $\int_0^t \exp\left\{-\int_s^t I(y_2,r)dr\right\}ds$ from \eqref{I_x2_t} -
numerical simulations suggest that this factor does not play a crucial role.
Finally, we further simplify the model by dropping the $x_2$ variable in the denominator. We end up with the following self-contained equation
\begin{align}\label{non-local-ODE1}
\partial_t G(x_2,t)=I(x_2,t)=\int_0^1\frac{(y_2-x_2)e^{-G}}{(y_2^2+e^{-2G})^2}dy_2,\quad G(x_2,0)=0.
\end{align}
Since there is no $x_1$ variable in the $G$ equation \eqref{non-local-ODE1} anymore, we use $x$ to denote $x_2$ and end up with the ODE model \eqref{non-local-ODE} with $K=1$. Noting that the right hand side of the model \eqref{non-local-ODE1} is linear with respect to $x_2$, we define the ansatz \[G(x,t)=:A(t)-B(t)x,\] and rewrite the model as follows
\begin{align}
A'(t)=&\int_0^1\frac{y_2e^{-G}}{(y_2^2+e^{-2G})^2}dy_2;\\
B'(t)=&\int_0^1\frac{e^{-G}}{(y_2^2+e^{-2G})^2}dy_2.
\end{align}
Here the growth of function $A$ encodes the potential blow-up while the growth of function $B$ suppresses it by thinning the vorticity filament.
 To study the stability of the blow-up suppression mechanics, we put parameter $\frac{1}{K}$ in front of the time evolution of the $B$ quantity and end up with:
\begin{subequations}\label{ABequation}
\begin{align}
A'(t)=&\int_0^1\frac{ye^{-G(y,t)}}{(y^2+e^{-2G(y,t)})^2}dy,\label{A'(t)}\\
B'(t)=&\frac{1}{K}\int_0^1\frac{e^{-G(y,t)}}{(y^2+e^{-2G(y,t)})^2}dy; \,\,\,\,A(0)=B(0)=0. \label{B'(t)}
\end{align} \end{subequations}
Notice that increasing $K$ weakens the nonlinearity depletion effect provided by growth of $B.$
This is equivalent to the model \eqref{non-local-ODE}.

Here is an outline of our approach to analysis of \eqref{ABequation}.
We will consider two regimes in the phase space $(A,B)\in\{A\leq \frac{1}{K}B\}$:
\begin{align}
\text{a) Initial Configuration: }\mathcal{I}:=& \left\{0 \leq A\leq 1, B\geq \frac{2}{K}\right\} \cup \left\{0\leq B\leq \frac{2}{K}\right\};\label{Initial_configuration}\\
\text{b) Final Configuration: }\mathcal{F}:=&\left\{A>1,\,\,\,A\leq \frac{K}{2}B\right\}.\label{Final_configuration}
\end{align}
Since $A(0)=B(0)=0$, the point $(A,B)$ initiates in the regime $\mathcal{I}$. Our first goal will be to show that from regime $\mathcal{I},$
the solution will necessarily transition to the regime $\mathcal{F}.$
A key observation is that the minimum $\min\{y,e^{-G(y,t)}\}$ plays the key role in the magnitude of the growth of the solutions. As a result, to understand the long time behavior of the solutions, we will propagate the bound on the profile $G$:
\begin{align}\label{Core_est}
y\leq e^{-G(y,t)},\quad \forall t\geq 0,\ \forall y\in[0,1].
\end{align}
Once the profile estimate \eqref{Core_est} is established, the behavior of the denominator near potential singularity $y=0$ will be controlled,
and this will lead to sufficiently strong estimates on the solution to show global regularity.
In fact, we will need a similar but stronger than \eqref{Core_est} bound to establish the asymptotic behavior of the solution.

The remaining part of the section is organized as follows: in Section 2.2, we will propagate the profile bound \eqref{Core_est} under the initial configuration $\mathcal{I}$; in Section 2.3, we focus on the final configuration and prove the global well-posedness of the equation \eqref{ABequation} and the asymptotic behavior.
\subsection{Initial configuration $\mathcal{I}$}

In this section, we focus on the initial configuration
\begin{align}
(A,B)\in \mathcal{I}.
\end{align}
As we discussed in the last section, the goal is to propagate the bound \eqref{Core_est} while $(A,B)\in \mathcal{I}$, as well as show that
solution will have to transition into the $\mathcal{F}$ regime.
First of all, the local existence of solutions follows from the classical ODE theorems, as the right hand side of the system is Lipschitz
in $A$ and $B$ in the neighborhood of the initial data. The resulting local solution is smooth in $t.$

\begin{lem}
Suppose that $K$ in \eqref{B'(t)} satisfies $1 \leq K \leq 1.3$. Then for all times $t>0$ that the solution $(A,B)$ remains in $\mathcal{I}$, we have
\begin{align}
e^{-G(y,t)} > & y,\quad \forall y\in[0,1],\label{Core_est_1}\\
A(t)<&\frac{K}{2}B(t).\label{Core_est_2} 
\end{align}
\end{lem}
\begin{proof}
Let us first show that \eqref{Core_est_1} and \eqref{Core_est_2} hold on some small initial time interval. 
Indeed, direct computation shows that $A'(0)=\frac14$ and $B'(0) = \frac{1}{4K} + \frac{\pi}{8K}.$ Therefore
\[ \lim_{t \rightarrow 0} \frac{A(t)}{B(t)} = \frac{A'(0)}{B'(0)} = \frac{2K}{2+\pi}. \]
This directly implies \eqref{Core_est_2} for $t \in (0,\delta_0)$ for some small $\delta_0>0.$
Next, note that \eqref{Core_est_1} holds for $t=0,$ with strict inequality for all $0\leq y <1.$ On the other hand,
$\partial_t G(1,t) = A'(t)-B'(t) <0.$ By continuity, the last two observations imply that $e^{-G(y,t)} > y$ for all $y \in [0,1]$
and $t \in (0, \delta_0)$ for some $\delta_0>0.$

Now suppose that $t_0$ is the supremum of times $t$ such that \eqref{Core_est_1} and $A(s),B(s) \in \mathcal{I}$ holds for all $0 < s \leq t.$
First let us show that \eqref{Core_est_2} also holds for $0<t \leq t_0$ in this case.
Consider
\begin{align}
F(y,t):=\frac{e^{-G(y,t)}}{(y^2+e^{-2G(y,t)})^2}.
\end{align}
Recalling that $\partial_y G(y,t) = -B(t) <0$ for all $t >0$ and $y \in [0,1]$ and
taking spacial derivative yields that
\begin{align*}
\pa_y F(y,t) =&\frac{e^{-G(y,t)}\bigg(-\partial_y G(y,t)(y^2-3e^{-2G(y,t)})-4y\bigg)}{(y^2+e^{-2G(y,t)})^3}<0,\quad \forall y \in [0,1],
\end{align*}
provided that $e^{-G(y,t)} \geq y.$
As a result of the monotonicity of $F$ and a symmetrization argument, we have that
\begin{align*}
A'(t)-\frac{K}{2}B'(t)=&\int_0^1\left(y-\frac{1}{2}\right)F(y,t)dy<0, \quad \forall t_0 \geq  t> 0,
\end{align*}
which yields the estimate \eqref{Core_est_2}.

Now we aim to show that $e^{-G(y,t_0)} > y$ for all $y \in [0,1],$ hereby contradicting the definition of $t_0$ unless $A(t),B(t) \notin \mathcal{I}$
for $t > t_0$ (notice that since $A(t),B(t)$ are monotone increasing, once the trajectory leaves $\mathcal{I}$ it never comes back).
The inequality \eqref{Core_est_1} is true for every $y> A/B$, $ \forall t>0$ thanks to the following argument: 
\begin{align*}
e^{-G(y,t)}=e^{-A(t)+B(t)y}=e^{B(t)(-\frac{A(t)}{B(t)}+y)}> e^{0} \geq y,\quad y> \frac{A}{B}.
\end{align*}
Next we prove the strict inequality $ e^{-G(y)}> y$ for $y\leq A/B$. In fact, we will prove the following stronger version
\begin{align}\label{Short_time_step_3_1}
L(y,t):=e^{-G(y,t)}-\frac{2}{K}y\geq 0
\end{align}
(it is stronger except for $y=0$ where we have \eqref{Core_est_1} simply because $A(t) < \infty$ while we are in $\mathcal{I}$).
When proving the inequality \eqref{Short_time_step_3_1}, we distinguish between two cases - either $\{B\leq \frac{2}{K}\}$, or $\{A\leq 1, B\geq \frac{2}{K}\}$.
Direct computation of $\partial_y L(y,t)$ 
yields the minimum point $y_{min}$
\begin{align}\label{ymin}
 y_{min}=\frac{A}{B}+\frac{1}{B}\log \frac{2}{K B}.
\end{align}
If $B\leq \frac{2}{K}$, $y_{min}$ is always to the right of $\frac{A}{B}$. Therefore it is enough to check the inequality at the point $y=A/B$. Applying the assumption \eqref{Core_est_2}, which we know holds up to and including $t_0,$ we have that
\begin{align*}
e^{-A+B\frac{A}{B}}-\frac{2}{K}\frac{A}{B}\geq 0.
\end{align*}
This completes the proof in the case $B\leq\frac{2}{K}$.
Next, we show that inequality \eqref{Short_time_step_3_1} holds when $\{A\leq 1, B\geq \CI\}$.
Plugging \eqref{ymin} this into the function $L$ and applying the assumption \eqref{Core_est_2} and $\{A\leq 1, B\geq 2/K\}$ yields
\begin{align*}
e^{-A+B(\frac{A}{B}+\frac{1}{B}\log \frac{2}{KB})}&-\CI\left(\frac{A}{B}+\frac{1}{B}\log \frac{2}{K B}\right)
=\frac{2}{K B}-\frac{2A}{K B}-\frac{2}{K B}\log \frac{2}{K B}\geq \frac{2}{KB}(1-A)\geq 0.
\end{align*}
This concludes the proof of the lemma.
\end{proof}

To understand the long time behavior, we need to first show that the solution $(A,B)$ initiating from the initial regime $\mathcal{I}$ must end up in the final regime $\mathcal{F}$ in a finite time. The following lemma addresses this issue.

\begin{lem}\label{Lem:end_of_the_initial_layer}
Consider the solutions to the equation  \eqref{ABequation} subject to the initial configuration. Then the solutions must reach the state $A=1, B>\frac{2}{K}$ at some time $\infty> t_0>0$.
\end{lem}
\begin{proof}
Since $A'(t)$ and $B'(t)>0$ for every $t \geq 0$ while the solutions are finite, 
the functions $A(t),B(t)$ are strictly positive and increasing for every $t>0.$ Therefore, for every $\delta>0,$
there exists a positive constant $c_1(\delta)>0$ such that the following inequalities are satisfied for $t \geq \delta$:
\begin{align*}
A'(t)\geq&\frac{1}{4}\int_0^1 ye^{3A-3By}dy=\frac{e^{3A}}{36B^2}(1-3Be^{-3B}-e^{-3B})\geq c_1\frac{e^{3A}}{B^2};\quad \frac{c_1e^{3A}}{B}\leq B'(t)\leq  \frac{e^{3A}}{3B}.
\end{align*}
Direct manipulation yields that
\begin{align*}
A'(t)\geq \frac{c_1}{2}(\log B^2)',\quad (B^2)'(t)\geq c_1e^{3A}/2\geq c_1/2.
\end{align*}
Therefore, we have that $A(t)\gtrsim \log t+C$, which in turn implies that $A(t)$ reaches $1$ at some time $t_0>0$. Since $A(0)=B(0)=0$ and by the previous lemma $A'(t)<\frac{K}{2}B'(t)$ for
all $t \leq t_0$, we have that if $A(t_0)=1$, then $B(t_0)>2/K$.
\end{proof}

\subsection{Final configuration $\mathcal{F}$}
In this section, we consider the long-time configuration, i.e.,
\begin{align}\label{long_time_configuration}
(A,B)\in \mathcal{F}.
\end{align}
Thanks to the Lemma \ref{Lem:end_of_the_initial_layer}, the solution $(A,B)$ of the model \eqref{ABequation} must end up in the final state $\mathcal{F}$. Therefore, it is enough to consider the following equations:
\begin{align}
A'(t)=&\int_0^1\frac{ye^{-G(y,t)}}{(y^2+e^{-2G(y,t)})^2}dy,\nb\\
B'(t)=&\frac{1}{K}\int_0^1\frac{e^{-G(y,t)}}{(y^2+e^{-2G(y,t)})^2}dy,\nb\\
A(t_0)=&1,\quad B(t_0)> 2/K, \quad K \in [1,1.3].\label{EQ:A-B_long_time}
\end{align}
It is easy to see that the solutions to \eqref{EQ:A-B_long_time} always continue in $\mathcal{F}$. We prove the following lemma:
\begin{lem}\label{lem:3_estimates}
Let $t_0$ be as in \eqref{EQ:A-B_long_time}. Consider the equation \eqref{EQ:A-B_long_time}.
The following two inequalities hold for all $t \geq t_0:$
\begin{subequations}\label{3_estimates}
\begin{align}
1)\ 1&-A(t)+\log \frac{KB(t)}{2} \geq 0;\label{3_estimates_1}\\
2)\ 1&-kA(t)+\log (kB(t))\geq 0, \quad {\rm for} \,\,\, 1<k<\frac{48(1-e^{-6/K})}{K(K^2+4)^2}.\label{3_estimates_2}
\end{align}
\end{subequations}
\end{lem}
\begin{rmk}
If $K=1$, then $1<k<\frac{48(1-e^{-6})}{25}$. Moreover, for $K \leq 1.3$, a simple computation shows that the constraint on $k$ in \eqref{3_estimates} is non-vacuous.
\end{rmk}
\begin{rmk}
Explicit calculation yields that the parameter $k$ is always strictly less than $2$ for any $K\in[1,1.3]$.
\end{rmk}

Before proving Lemma~\ref{lem:3_estimates}, we make the following observation.

\begin{lem}
If the inequalities \eqref{3_estimates} are satisfied at some time $t$, then the following profile bounds hold for all $y \in [0,1]$:\begin{subequations}
\begin{align}
1)\ &y\leq \CC e^{-G(y,t)}; \label{bound_1}\\
2)\ & y\leq e^{-kG(y,t)}. \label{bound_2}
\end{align}
\end{subequations}
\end{lem}
\begin{proof} First, we prove inequality  \eqref{bound_1}. Recalling that $G(y,t)=A(t)-B(t)y$, we need to show
\begin{align*}
\frac{K}{2} e^{-A+B y}- y\geq 0.
\end{align*}
Evaluating at the minimal point $y_{min}=\frac{A}{B}+\frac{1}{B}\log\frac{2}{K B}$, we have that the following estimate implies the inequality \eqref{bound_1}
\begin{align*}
\CC\frac{2}{K B}-\frac{A}{B}+\frac{1}{B}\log{\frac {KB}{2}}\geq 0.
\end{align*}
After simplification, this is the same as the inequality \eqref{3_estimates_1}. The inequality \eqref{bound_2} is equivalent to
\begin{align}
e^{-kA+kBy}-y\geq 0. \label{aux7301}
\end{align}
The expression in \eqref{aux7301} has minimizer $y_{min}=\frac{A}{B}+\frac{1}{kB}\log\frac{1}{kB}.$  
 If we substitute this value into the expression, we obtain
\begin{align*}
\frac{1}{kB}-\frac{A}{B}-\frac{1}{kB}\log \frac{1}{kB}\geq 0\Leftrightarrow 1-kA+\log (kB)\geq 0,
\end{align*}
which is the same as \eqref{3_estimates_2}. This concludes the proof of the lemma.
\end{proof}

\begin{proof}[Proof of Lemma \ref{lem:3_estimates}]Direct calculation yields that the inequalities \eqref{3_estimates} hold at the initial time $t_0$. To propagate the inequalities \eqref{3_estimates} for all time, it is enough to prove the following strict derivative upper bounds for all $t \geq t_0$
assuming the estimates \eqref{3_estimates} for  $\forall s\in[t_0, t]:$ 
\begin{subequations}
\begin{align}
A'(t)<&(\log B)'(t),\label{A'lnB'}\\
kA'(t)<&(\log B)'(t). 
\label{kA'lnB'}
\end{align}
\end{subequations}
Since $k>1,$ it suffices to prove the stronger bound \eqref{kA'lnB'}.

To prove the inequality \eqref{kA'lnB'}, we provide an upper bound of $A'$ and a lower bound of $B'$. The estimate of $A'$ is carried out as follows:
\begin{align}
A'(t)\leq&\int_0^1 ye^{3G(y,t)}dy=\int_0^1 y e^{3A-3By}dy
=\frac{e^{3A}}{9B^2}\left(1-3Be^{-3B}-e^{-3B}\right)\leq \frac{e^{3A}}{9B^2}. 
\label{A'_upper_estimate}
\end{align}

Applying the assumption that (\ref{3_estimates}) holds for all $s\in [t_0,t]$ (and so, in particular, for $s=t$), \eqref{bound_1}, and the initial data $A(t_0)=1$, $ B(t_0)>2/K$, we estimate $B'(t)$ as follows:
\begin{align}
B'(t)=\frac{1}{K}\int_{0}^1\frac{e^{-G(t)}}{(y^2+e^{-2G(t)})^2}dy\geq& \frac{1}{K}\int_{0}^1\frac{e^{-G}}{(\frac{K^2}{4}e^{-2G}+e^{-2G})^2}dy
\geq\frac{1}{K(K^2/4+1)^2}\frac{e^{3A}}{3B}(1-e^{-3B}).\label{Bt_lower_bound_median_time}
\end{align}
We summarize the bounds \eqref{A'_upper_estimate} and \eqref{Bt_lower_bound_median_time} as follows:
\begin{subequations}\begin{align}
A'(t)\leq &\frac{e^{3A}}{9B^2};\label{Pinching_of_A'(t)}\\
\frac{e^{3A}}{3B}\frac{1}{K(K^2/4+1)^2}& \left(1-e^{-6/K}\right)\leq B'(t).\label{Pinching_of_B'(t)}
\end{align}
\end{subequations}
Direct calculation yields that
\begin{align*}
A'(t)\leq \frac{e^{3A}}{9B^2}\leq (\log B)'\frac{K(K^2+4)^2}{48(1-e^{-6/K})}<(\log B)'
\end{align*}
for $1\leq K\leq 1.3$. Therefore the inequalities  \eqref{A'lnB'} and \eqref{kA'lnB'} hold true for $1<k<\frac{48(1-e^{-6/K})}{K(K^2+4)^2}$. This completes the proof of the lemma. 
\end{proof}

Now we are ready to prove Theorem~\ref{Thm_1}.
\begin{proof}[Proof of Theorem~\ref{Thm_1}]
Let us denote by $p(t)$ the point where the function $G(y,t)=A(t)-B(t)y$ achieves the value $A(t)/2$.
Note that $p(t)$ is unique and \eqref{3_estimates_1} implies that $0<p(t)<1.$
We estimate $B'(t)$ from below using \eqref{bound_2}:
\begin{align}
B'(t)=& \frac{1}{K}\left(\int_0^{p}+\int_{p}^1\right) \frac{e^{-G(y,t)}}{(y^2+e^{-2G(y,t)})^2}dy\geq  \frac{1}{K}\int_0^{p}\frac{e^{-G}}{(e^{-2kG}+e^{-2G})^2}dy+\frac{1}{K}\int_{p}^1 \frac{e^{-G}}{(y^2+e^{-2G})^2}dy \nb\\
=& \frac{1}{K}\int_0^{p}\frac{e^{3G}}{(e^{-2(k-1)G}+1)^2}dy+\frac{1}{K}\int_{p}^1 \frac{e^{-G}}{(y^2+e^{-2G})^2}dy\nb\\
\geq&{\frac{1}{K(1+e^{-(k-1)A})^2}}\int_0^{p} e^{3G}dy
\geq \frac{1}{K(1+e^{-(k-1)A})^2}\frac{e^{3A}}{3B}-\frac{1}{K(1+e^{-(k-1)A})^2} \frac{e^{3A/2}}{3B}.\label{B'(t)_lower_bound}
\end{align}
Applying a similar argument yields that
\begin{align}
A'(t)\geq &\int_0^{p} \frac{ye^{3G}}{(e^{-(k-1)A}+1)^2}dy
=\frac{e^{3A}}{(e^{-(k-1)A}+1)^2 (3B)^2}(1-3Bp e^{-3A/2}-e^{-3A/2}).\label{A'_lower_bound}
\end{align}
An upper bound for $B'$ is derived as follows:
\begin{align}
B'(t)=& \frac{1}{K}\int_0^{1}\frac{e^{-G(y,t)}}{(y^2+e^{-2G(y,t)})^2}dy
\leq \frac{1}{K}\int_0^{1}e^{3G}dy\leq\frac{e^{3A}}{3KB}.\label{B'(t)_upper_bound}
\end{align}
Summarizing the estimates \eqref{B'(t)_lower_bound}, \eqref{A'_lower_bound}, \eqref{B'(t)_upper_bound} as well as the estimates \eqref{Pinching_of_A'(t)},  \eqref{Pinching_of_B'(t)}, we have the bounds
\begin{subequations}\begin{align}
\frac{e^{3A}}{(e^{-(k-1)A}+1)^2 (3B)^2}(1-3Bp e^{-3A/2}-e^{-3A/2})& \leq A'(t)\leq \frac{e^{3A}}{9B^2};\label{Pinching_of_A'(t)_1}\\
\frac{1-e^{-3A/2}}{(1+e^{-(k-1)A})^2}\frac{e^{3A}}{3KB}
&\leq B'(t)\leq \frac{e^{3A}}{3KB}.\label{Pinching_of_B'(t)_1}
\end{align}
\end{subequations}
It follows that
\begin{align*}
\left(1-\frac{2e^{-(k-1)A}+e^{-2(k-1)A}+e^{-3A/2}}{(1+e^{-(k-1)A})^2}\right)A'(t)\leq  \frac{KB'}{3B},
\end{align*} which in turn yields that
\begin{align*}\left(A-\frac{4}{(k-1)(1+e^{-(k-1)A})}\right)'\leq  \frac{K(\log B)'}{3}.
\end{align*}
Direct integration yields that 
there exists a constant $C$ depending only on $K$ such that 
\begin{align*}
A\leq\frac{K}{3}\log B+C,
\end{align*}
leading to
\begin{align}\label{expA_upper_bound}
e^A\leq e^C B^{\frac{K}{3}}.
\end{align}
Combining it with the $B'$ estimate  \eqref{Pinching_of_B'(t)_1} yields
\begin{align*}
 B'(t)\leq \frac{B^{K}}{3KB}e^{2C},
\end{align*}
so $B(t)$ satisfies the upper bound
\begin{align}\label{upper_bound_of_B}
B(t)\lesssim t^{1/(2-K)}, \quad \forall t\geq 1. 
\end{align} Combining this and $A'(t)\leq K B'(t)$, we obtain that $A(t)$ cannot blow up in finite time. We conclude that the solution to the equation \eqref{EQ:A-B_long_time} is bounded for any finite time. This completes the proof of the global boundedness of the solution.

To derive the long time asymptotic behavior, we estimate the lower bound of $A$ and $B$. From \eqref{Pinching_of_A'(t)_1} and \eqref{Pinching_of_B'(t)_1}
we obtain that
\begin{align*}
A'(t)\geq \frac{e^{3A}}{(3B)^2(1+e^{-(k-1)A})^2}\left(1-\frac{3}{2}Ae^{-3A/2} -e^{-3A/2}\right)\geq \frac{K(\log B)'}{3}
\frac{1-\frac{3}{2}Ae^{-3A/2} -e^{-3A/2}}{(1+e^{-(k-1)A})^2}.
\end{align*}
A calculation shows that for all $A$ large enough (say $A\geq \frac{2}{k-1}$), we have
\begin{align*}
\left(A+\frac{1}{k-1}2\log (1-5e^{-(k-1)A})\right)'\geq\frac{K(\log B)'}{3}.
\end{align*}
Here the the fact that $k\in[1,2)$ \eqref{3_estimates_2} is used.
Therefore there exists a universal constant $C$ such that \begin{align*}
A(t)\geq \frac{K}{3}\log B-C,
\end{align*}
which is equivalent to
\begin{align}\label{expA_lower_bound}
e^A\geq e^{-C}B^{\frac{K}{3}}.
\end{align}
Plugging this into the lower bound for $B'$ \eqref{Pinching_of_B'(t)_1}, we obtain that for all $B$ big enough, there exists universal constants $C_1,C_2,C_3,C_4$ such that
\begin{align*}
B'(t)\geq\frac{1-C_1e^{-K/2\log B}}{(1+C_2e^{-(k-1)\frac{K}{3}\log B})^2}B^{K-1}C_3=\frac{1-C_1B^{-K/2}}{(1+C_2 B^{-(k-1)K/3})^2}B^{K-1}C_3\geq C_4B^{K-1}.
\end{align*}
Therefore, we have that for all $t \geq 1,$
\begin{align}\label{B_lower_bound}
B\gtrsim t^{1/(2-K)}.
\end{align}
Combining \eqref{upper_bound_of_B}, \eqref{B_lower_bound}, \eqref{expA_lower_bound} and \eqref{expA_upper_bound}, we obtain the long time asymptotic behavior of $(A,B)$. This concludes the proof.
\end{proof}

\section{A Higher Order Boundary Layer Model of the Hou-Luo Scenario}

\subsection{The Setting}\label{hlblsetting}

Recall that here we will consider the following model:
\begin{subequations}\label{Higher_order_model1}
\begin{align}
\pa_t \omega+u_1\pa_{x_1}\omega&=\frac{\rho}{x_1},\label{vorticity_equation1}\\
\pa_t \rho+u_1\pa_{x_1}\rho&=0,\label{mass_transport_eq1}\\
u_1(x_1,x_2,t)&=-x_1\iint_{[0,\infty)^2}\frac{y_1y_2\omega(y)}{|y|^4}dy+x_2\omega(x_1,0,t).
\end{align}
\end{subequations}
Local well-posedness for \eqref{Higher_order_model} for the compactly supported away from the origin initial data $(\omega_0,\rho_0) \in H^s \times H^s$
is not hard to show - the argument is standard and will be omitted. A very similar argument can be found, for instance, in \cite{KT1}.

We begin by writing the system \eqref{Higher_order_model1} in Lagrangian coordinates. Define the flow map $\Phi_t(x_1^0,x_2)$ by
\begin{align}\label{trajbl}
\frac{d\Phi^1_t(x_1^0,x_2)}{dt}=u_1(\Phi_t^1(x_1^0,x_2),x_2,t),\quad \frac{d}{dt}\Phi_t^2(x_1^0,x_2)\equiv 0,\ \Phi_{t=0}(x_1^0,x_2)=(x_1^0,x_2).
\end{align}
We apply $(x_1^0,x_2)$ to denote the initial point of the flow-map $\Phi_t$, and $(x_1,x_2)$ to denote the current location, respectively.
To simplify the notation, we define the following quantities $J(t), D(t), Q(t), H(t)$:
\begin{align}
J(t):=&\iint_{[0,\infty)^2}\frac{y_1y_2\omega(y,t)}{|y|^4}dy;\label{Defn-I}\\
D(t):=&e^{\int_0^t J(s)ds}\label{Defn-G};\\
Q(t):=&\int_0^t D(s)ds=\sqrt{H(t)};\label{Defn:Q}\\
H(t):=&\left(\int_0^t D(s)ds\right)^2.\label{Defn-H}
\end{align}
Here the $Q(t)$ is the blow-up variable that, as we will see, reflects growth of vorticity. The strictly increasing quantity $H(t)$ can be viewed as an adapted time variable before the blow-up.
The equation \eqref{trajbl} can be explicitly written as
\begin{align}\label{Flow-map-ODE}
\frac{d}{dt}\Phi_t(x_1^0,x_2)=(-\Phi_t^1(x_1^0,x_2)J(t)+x_2\omega(\Phi_t^1(x_1^0,{x_2}),{0},t),0),\quad \Phi_{t=0}(x_1^0,x_2)=(x_1^0,x_2).
\end{align}
The mass transport equation \eqref{mass_transport_eq1} implies that the density $\rho$ is conserved along the flow characteristics $\Phi_t(x_1^0,x_2)$ before the blow-up time $T_\star:$
\begin{align*}
\rho(\Phi_t^1(x_1^0,x_2),x_2,t)=\rho_0(x_1^0,x_2),\quad \forall t\in [0,T_\star).
\end{align*}
On the other hand, by the vorticity equation \eqref{vorticity_equation1}, the vorticity evolves along the trajectories as follows:
\begin{align*}
\frac{d\omega(\Phi_t^1(x_1^0,x_2),x_2,t)}{dt}=&\frac{\rho(\Phi_t^1(x_1^0,x_2),x_2)}{\Phi_t^1(x_1^0,x_2)}=\frac{\rho_0(x_1^0,x_2)}{\Phi_t^1(x_1^0,x_2)}.
\end{align*}
Collecting all the equations above, we obtain the model \eqref{Higher_order_model1} in Lagrangian coordinate:
\begin{align}
\left\{\begin{array}{rrrr}\ba
\frac{d}{dt}\Phi_t^1(x_1^0,x_2)
=-\Phi_t^1(x_1^0,x_2)\iint_{[0,\infty)^2}\frac{y_1y_2 \omega(y,t)}{|y|^4}dy+x_2\omega(\Phi_t^1(x_1^0,{x_2}),{0},t),\\
\frac{d\omega(\Phi_t^1(x_1^0,x_2),x_2,t)}{dt}= \frac{\rho_0(x_1^0,x_2)}{\Phi_t^1(x_1^0,x_2)},\quad
\rho(\Phi_t^1(x_1^0,x_2),x_2,t)=\rho_0(x_1^0,x_2),\\
\omega_0\equiv 0,\quad \frac{\rho_0(x_1^0,x_2)}{x_1^0}=\varphi_\delta(x_1^0)\eta(x_2). \ea
\end{array}\right.\label{Basic_Model}
\end{align}
As pointed out above, our initial data will have $\omega_0 \equiv 0,$ while the initial density $\rho_0(x_1,x_2)$ will be compactly supported and equal to the product of $x_1 \varphi(x_1)$ and $\eta(x_2).$
Each factor is defined as follows. For the $x_1$ component, we have that
\begin{align}
\varphi_\delta(x_1):=\left\{\begin{array}{rr}\ba
{1}/\delta,&\quad x_1\in[\delta,L\delta<1],\\
0,&\quad x_1\in [0,{\delta}/{2}]\cup [(L+1)\delta,\infty),\\
C^{\infty},&\quad x_1 \text{ in other regions}.\ea\end{array}\right. \label{varphi_delta}
\end{align}
The parameter $L$ is a large parameter and $\delta\ll{1}$ is a small parameter that will be chosen later.
It is convenient for us to have $L\delta<1,$ and any choice of $L,$ $\delta$ we make will be done to ensure that.
The smooth decreasing $x_2$-factor $\eta(x_2)\in C_c^\infty(\rr_+)$ is defined as follows:
\begin{align}
\eta(x_2)=\left\{\begin{array}{rr} 1,&\quad x_2\in [0,1];\\
0,&\quad x_2\geq 2;\\
C^\infty,&\quad x_2\in[1,2].\end{array}\right.
\end{align}

\subsection{Preliminary Calculations}\label{Sec:Formula}

Now we prepare the necessary estimates for the blow-up argument, deriving an effective formula for the flow map $\Phi_t^1(x_1^0,x_2)$ and a sufficiently strong lower bound for the vorticity $\omega(x_1^0,x_2,t)$.

First we calculate the flow map $\Phi_t^1(x_1^0,0)$ on the $x_1$-axis, i.e., $x_2\equiv0$. Solving the flow-map ODE \eqref{Flow-map-ODE} and simplifying the result by adopting the notation of $D$ \eqref{Defn-G} yield that
\begin{align}
\Phi_t^1(x_1^0,0)=&x_1^0 e^{-\int_0^t J(s)ds}=x_1^0D(t)^{-1},\ \
\frac{\Phi_t^1(x_1^0,0,t)}{x_1^0}=D(t)^{-1},\quad \forall \,\, x_1^0\in\rr_+.\label{pax1_Phit_x2=0}
\end{align}

Calculation of the flow map $\Phi_t^1(x_1^0,x_2)$ for general $x_2\neq 0$ involves the boundary vorticity $\omega(\Phi_t^1(x_1^0,x_2),0,t)$.
It is worth mentioning that the second coordinate of the boundary vorticity, i.e., $0$, and the second coordinate of the flow-map involved, i.e., $x_2\neq 0$, do not match. Hence 
to determine the boundary vorticity $\omega(\Phi_t^1(x_1^0,x_2),0,t)$, we track the preimage of the current position $\Phi_t^{1}(x_1^0,x_2)$ under the flow-map along the $x_1$-axis. 
 We combine the equation for $\frac{\Phi_t^1(x_1^0,0,t)}{x_1^0}$ \eqref{pax1_Phit_x2=0} and the second equation and the fourth equation in \eqref{Basic_Model} to calculate the  vorticity along the characteristic  initiated from the position $(x_1^\star,0)$:
\begin{align}\label{Boundary_omega}
\omega(\Phi_t^1(x_1^\star, 0),0,t)=&\int_0^t\frac{\rho_0(x_1^\star,0)}{\Phi_s^1(x_1^\star,0)}ds
=\frac{\rho_0(x_1^\star,0)}{x_1^\star}\int_0^tD(s)ds=\varphi_{\delta}(x_1^\star)\int_0^tD(s)ds.
\end{align}
Here the notation $x_1^\star$ is introduced to avoid confusion later.
Next we  determine the preimage of the current $x_1$-projection $(\Phi_t^1(x_1^0,x_2),0)$ under the flow-map on the $x_1$-axis. Direct application of the relation \eqref{pax1_Phit_x2=0} yields that the inverse $\Phi_t^{-1} $ of the flow-map on the $x_1$-axis can be expressed as
\begin{align}
\Phi_t^{-1}\left(\Phi_t^1(x_1^0,x_2),0\right)=\left(D(t)\Phi_t^1(x_1^0,x_2),0\right)=:\left({U}(t,x_1^0,x_2),0\right).\label{Defn-U}
\end{align}
We call $U(t,x_1^0,x_2)$ the ``back-to-label map". Combining the equation of the boundary vorticity  \eqref{Boundary_omega} and the `back-to-label map' \eqref{Defn-U}, we obtain the explicit expression of the boundary vortcity $\omega(\Phi_t^1(x_1^0,x_2),0,t)$:
\begin{align}
\omega(\Phi_t^1(x_1^0,x_2),0,t)=&\omega\left(\Phi_t^1\left(\Phi
_t^{-1} (\Phi_t^1(x_1^0,x_2),0),0\right),0,t\right)\\
=&\frac{\rho_0\left(\Phi_t^{1}(x^0_1,x_2)D(t),0\right)}{{\Phi_t^1(x_1^0,x_2)D(t)}}\int_0^tD(s)ds.\label{Boundary_omega_backtracked}
\end{align}
Combining the explicit form of the boundary vorticity  \eqref{Boundary_omega_backtracked}, the definition of $U$ \eqref{Defn-U}  and the basic model \eqref{Basic_Model} yields that
\begin{align*}
\frac{d}{dt}\Phi_t^1(x_1^0,x_2)=-\Phi_t^1(x_1^0,x_2)J(t)+ x_2\frac{\rho_{0}(D(t)\Phi_t^1(x_1^0,x_2),0)}{D(t)\Phi_t^1(x_1^0,x_2)}\int_0^tD(s)ds.
\end{align*}
This equation can be simplified with  the definitions of $U$ \eqref{Defn-U} and of $D, H$ \eqref{Defn-G}, \eqref{Defn-H}:
\begin{align}
\frac{d}{dt}&\left(\Phi_t^1(x_1^0,x_2)D(t)\right)=x_2\frac{\rho_{0}(D(t)\Phi_t^1(x_1^0,x_2),0)}{D(t)\Phi_t^1(x_1^0,x_2)}\frac{1}{2}\frac{d}{dt}\left(\int_0^tD(s)ds\right)^2 \nonumber \\
&\Leftrightarrow \nonumber \\
\frac{d}{dt}&U(t,x_1^0,x_2)=\frac{x_2}{2}\frac{\rho_0(U(t,x_1^0,x_2),0)}{U(t,x_1^0,x_2)}\frac{d}{dt}H(t).\label{U-H}
\end{align}
Since by \eqref{varphi_delta}, we have that $\rho_0(x_1^0,0)/x_1^0=\frac{1}{\delta}$  on the interval $[\delta, L\delta]$,
it follows that if $U(s,x_1^0,x_2) \in [\delta, L\delta]$ for all $0 \leq s \leq t$ then
\begin{align}
 U(t,x_1^0,x_2)=x_1^0+\frac{x_2}{2\delta }H(t).\label{U_middle_range}
\end{align}
Thus if $y_1^0 \in [\delta, L \delta]$ and while $U(s,y_1^0,y_2)$ remains in this interval,
we can integrate the vorticity equation in \eqref{Basic_Model} and obtain that
\begin{align}
\omega(\Phi_t^1(y_1^0,y_2),y_2,t)=\rho_0(y_1^0,y_2)\int_0^t \frac{D(s)}{y_1^0+\frac{y_2}{2\delta}H(s)}ds.\label{Vorticity_mf2}
\end{align}
To obtain a useful lower bound on the vorticity, we will initially make the following assumption on the range of the original labels to watch, to be refined later:
\begin{align}\label{goodic}
\delta\leq y_1^0\leq L\delta/2, \quad y_2\leq \frac{\delta^2 L}{H(t)}.
\end{align}
Using \eqref{U_middle_range} we find that if \eqref{goodic} holds, then $\delta \leq y_1^0 \leq U(s,y_1^0,y_2) \leq L\delta$ for all $0 \leq s \leq t.$
Under this range constraint, we calculate that (recall that $Q(t)= \sqrt{H(t)} = \int_0^t D(s)\,ds$):
\begin{align*}
\omega(\Phi_t^1(y_1^0,y_2),y_2,t)=\frac{\rho_0(y_1^0,y_2)}{y_1^0}\int_0^{Q(t)} \frac{1}{1+(y_1^0)^{-1}\frac{y_2}{2\delta}Q^2(s)}d\underbrace{Q(s)}_{\int_0^s D(\tau)d\tau}= \\
\frac{\rho_0(y_1^0,y_2)}{y_1^0}\sqrt{\frac{2\delta y_1^0}{y_2}}\arctan\left(\sqrt{\frac{y_2}{2\delta y_1^0}}\int_0^t D(s)ds\right).
\end{align*}
Now we define the rescaled spacial variable \begin{align}\label{Y_2}
Y_2:=\frac{y_2}{\delta^2}H(t).
\end{align}
With the $Y_2$-variable, we represent the above expression for vorticity as
\begin{align}\label{vorteq926}
\omega(\Phi_t^1(y_1^0,y_2),y_2,t)=\frac{1}{\delta}\int_0^t D(\tau)d\tau\frac{\arctan\left(\sqrt{\delta Y_2/2y_1^0}\right)}{\sqrt{{\delta Y_2}/{2y_1^0}}}
\end{align}
if \eqref{goodic} holds.
To derive a suffucuently strong lower bound for vorticity, we will further restrict the set of labels that we monitor. Namely, let us require in addition that
\begin{align}\label{arctan926}
\frac{\arctan\left(\sqrt{\delta Y_2/2y_1^0}\right)}{\sqrt{{\delta Y_2}/{2y_1^0}}}\geq\frac{\arctan(1)}{1}=\frac{\pi}{4}.
\end{align}
Since the function $\frac{\arctan x}{x}$ is monotone decreasing in $x$ on $[0,\infty),$ for \eqref{arctan926} to hold it suffices enhance the first condition in \eqref{goodic} to
\begin{align}
\max\left\{\frac{1}{2}Y_2\delta,\delta\right\}\leq y_1^0\leq \frac{L\delta}{2},\quad 0 \leq y_2\leq \frac{\delta^2 L}{H(t)}. \label{Domain}
\end{align}
Applying the equation \eqref{U_middle_range}, the definition of $U$ \eqref{Defn-U}, and the fact that $U(s,y_1^0,y_2) \in [\delta, L\delta]$ for $0 \leq s \leq t$ for the initial conditions
satisfying \eqref{Domain}, we find that
\begin{align}
\Phi_t^1(y_1^0,y_2)\in \bigg[\max\left\{\frac{1}{2}Y_2\delta,\delta\right\}D(t)^{-1}+\frac{y_2}{2\delta}H(t)D(t)^{-1},\frac{L\delta}{2}D(t)^{-1}+\frac{y_2}{2\delta}H(t)D(t)^{-1}\bigg]=[Lo,Up].\label{Lo_Up}
\end{align}
Combining the above calculations with \eqref{vorteq926}, we conclude that for $y_1 \in [Lo,Up],$ $0 \leq y_2 \leq \frac{\delta^2 L}{H(t)}$, the vorticity satisfies
\begin{align}
\omega(t, y_1,y_2)\geq \frac{\pi}{4}\frac{\int_0^tD(\tau)d\tau}{\delta}. \label{lower_bound_of_omega}
 \end{align}

\subsection{Finite time blow up}\label{ftbu}
Now we are ready to derive a differential inequality that will lead to the finite time blow up argument.

Let us define
\begin{align}
E(t):= \delta^{-1} D(t)^{-1}H(t). \label{defn_E(t)}
\end{align}
Since $D(0)=1$ while $H(0)=0,$ we have that for some initial time period $E(t) \leq 1.$
Our first goal will be to show that we can choose $\delta$ and $L$ so that
\begin{align}\label{Ebound926}
E(t) \leq 2
\end{align}
for all times while the solution remains regular. Specifically, we have the following
\begin{lem}\label{propfin}
Let us choose any $L$ and $\delta$ such that
\begin{align}\label{choicedeltaL}
L > 5 e^{\frac{96}{\pi}}, \,\,\, 0 < \delta L \leq 1.
\end{align}
Then $E(t) \leq 2$ for all times $t$ while the solution remains regular.
\end{lem}
\begin{proof}
Suppose, on the contrary, that there exists $t_1>0$ such that $E(t_1)=2$ while $E(t) <2$ for all $0 \leq t < t_1.$
Our goal is to show that in this case
\begin{align}
\frac{1}{\delta}H'(t_1)<D'(t_1). \label{Goal}
\end{align}
This would imply that $E(t)>2$ for some $t<t_1$ furnishing a contradiction.
Now to prove \eqref{Goal}, we calculate a lower bound for
\begin{align}\label{I}
D(t_1)^{-1}D'(t_1)=J(t_1)=&\int_{y_2=0}^{\infty}\int_{y_1=0}^\infty\frac{y_1y_2}{|y|^4}\omega(t_1,y)dy_1dy_2.
\end{align}
First, we observe that
\begin{equation}\label{caseA}
\frac{\delta^2 L}{H(t_1)} \leq 1.
\end{equation}
 Indeed, if $\frac{\delta^2 L}{H(t_1)} > 1$ then
$\delta^{-1} H(t_1) \leq \delta L \leq 1,$
while $D(t) \geq 1$ for all $t$ while solution exists (this follows from positivity of vorticity in the first quadrant and the definition of $D(t)$ \eqref{Defn-G}).
Thus we get $E(t_1) \leq 1,$ a contradiction.

Since vorticity is non-negative, we have that
\begin{align}
J(t_1)\geq \int_{0}^{\frac{\delta^2L}{H(t_1)}}\int_{Lo}^{Up}\frac{y_1y_2}{|y|^4}\omega(t_1,y_1,y_2)dy_1dy_2.\label{Region_of_interest}
\end{align}
Using \eqref{lower_bound_of_omega}, the definition of $E(t)$ \eqref{defn_E(t)}, and the definition of time $t_1,$ we estimate
\begin{align} \nonumber
J(t_1) \geq \frac{\pi}{8}\int_{0}^{\frac{\delta^2 L}{H(t_1)}}\left(\frac{y_2}{Lo^2+y_2^2}-\frac{y_2}{Up^2+y_2^2}\right)\frac{\int_0^{t_1} D(s)ds}{\delta}dy_2\\ \nonumber
\geq \frac{\pi}{16 \delta}\int_{0}^{\frac{\delta^2 L}{H(t_1)}}\frac{\int_0^{t_1} D(s)ds \left((L/2)^2-\max\left\{\frac{1}{2}Y_2,1\right\}^2\right)\delta^2 D(t_1)^{-2}y_2\, dy_2}{(\max\left\{\frac{1}{2}Y_2,1\right\}^2\delta^2 D^{-2}(t_1)+y_2^2(1+\frac{1}{4}E(t_1)^2)(\frac{\delta^2 L^2}{4}D(t_1)^{-2}+y_2^2(1+\frac{1}{4}E(t_1)^2)}\\
=\frac{\pi}{16}\frac{\int_0^{t_1} D(s)ds}{\delta}\int_{0}^{\frac{\delta^2 L}{H(t_1)}}\left(\frac{1}{\max\left\{\frac{1}{2}Y_2,1\right\}^2\delta^2 D(t_1)^{-2}+2 y_2^2}
-\frac{1}{\frac{\delta^2L^2}{4} D(t_1)^{-2}+ 2 y_2^2}\right)y_2dy_2. \nonumber
\end{align}
Now we apply the definition \eqref{defn_E(t)} of $E(t)$  and of $Y_2$-variable \eqref{Y_2} to rewrite the integral, and direct integration yields
\begin{align}
J(t_1)\geq&\frac{\pi}{16}\frac{\int_0^{t_1} D(s)ds}{\delta}\left(\int_{0}^2+\int_{2}^{L}\right)\left(\frac{1}{\max\left\{\frac{1}{2}Y_2,1\right\}^2+ 2 Y_2^2}-\frac{1}{ L^2+2 Y_2^2}\right)Y_2dY_2 \nonumber \\ 
\geq&\frac{\pi}{16}\frac{\int_0^{t_1} D(s)ds}{\delta}\int_{Y_2=2}^{L}\left(\frac{1}{3 Y_2}-\frac{Y_2}{L^2+2 Y_2^2}\right)dY_2 \nonumber \\
\geq &\frac{\pi}{16}\frac{\int_0^{t_1} D(s)ds}{\delta}\left(\frac{1}{3}\log \frac{L}{2 }+\frac{1}{4}\log\left(\frac{L^2+8}{3L^2}\right)\right) \geq \frac{\pi}{48}\frac{\int_0^{t_1} D(s)ds}{\delta} \log \frac{L}{5}. \label{lbfin926}
\end{align}
Due to the choice of the parameter $L$ \eqref{choicedeltaL}, it follows that
\[ J(t_1) > \frac{2}{\delta} \int_0^{t_1} D(s)ds, \]
and this implies \eqref{Goal}.
\end{proof}



Finally, we prove finite time blow up and thus Theorem~\ref{Thm2}.
\begin{proof}[Proof of Theorem \ref{Thm2}]
Let us apply the relation \eqref{Ebound926} to obtain the following differential inequality for $Q(t)=\sqrt{H(t)}:$
\begin{align}\label{diffineqftbu926}
Q'(t)=D(t)\geq \frac{1}{2\delta}H(t)= \frac{1}{2\delta}Q(t)^2.
\end{align}
However, since $Q(\ep)=\int_0^\ep D(s)ds$ and  $D(s) \geq 1$ for all $s$ while solution exists, we have that $Q(t)>0$ is positive for any small time $t =\ep >0$.
Considering the differential inequality \eqref{diffineqftbu926} for $t\geq\ep$, we see that the solution $Q(t)$ blows up at a finite time $T_\star <\infty$
which by \eqref{lower_bound_of_omega} implies that $\lim_{t \rightarrow T_\star} \|\omega(\cdot, t)\|_{L^\infty} = \infty.$
\end{proof}

{\bf Acknowledgement}. \rm
The authors acknowledge partial support of the NSF-DMS grants 1848790, 2006372 and 2006660.

\bibliographystyle{abbrv}


\end{document}